\pdfoutput=1
\documentclass[12pt]{amsart}
\usepackage{amsmath,amssymb,tikz,pgfplots,mathrsfs,framed,pgf,latexsym,amsthm,footnpag,url,multicol}
\usepackage{mathtools}
\usetikzlibrary{cd}
\usepackage[colorlinks,breaklinks=true]{hyperref}
\usepackage[figure,table]{hypcap}
\hypersetup{
bookmarksnumbered,
pdfstartview={FitH},
citecolor={black},
linkcolor={black},
urlcolor={black},
pdfpagemode={UseOutlines}
}
\usepackage{cleveref}
\usepackage{relsize}
\usepackage[bbgreekl]{mathbbol}
\DeclareSymbolFontAlphabet{\mathbbl}{bbold}
\DeclareSymbolFontAlphabet{\mathbb}{AMSb}

\newtheorem{theorem}{Theorem}
\crefname{theorem}{Theorem}{Theorems}
\newtheorem{lemma}[theorem]{Lemma}
\crefname{lemma}{Lemma}{Lemmas}

\crefname{corollary}{Corollary}{corollaries}
\newtheorem{proposition}[theorem]{Proposition}
\crefname{proposition}{Proposition}{Propositions}
\theoremstyle{definition}
\newtheorem{definition}[theorem]{Definition}
\crefname{definition}{Definition}{Definitions}
\newtheorem{remark}[theorem]{Remark}
\crefname{remark}{Remark}{Remarks}

\crefname{example}{Example}{Examples}

\crefname{conjecture}{Conjecture}{Conjectures}

\crefname{question}{Question}{Questions}
\numberwithin{equation}{theorem}

\crefname{claim}{Claim}{Claims}

\renewcommand{\inf}{\mathrm{inf}}
\newcommand{\AKp}{\mathbb{A}_K^{+}}
\newcommand{\AK}{\mathbb{A}_K}
\newcommand{\AFp}{\mathbb{A}_F^{+}}
\newcommand{\AF}{\mathbb{A}_F}
\newcommand{\ALp}{\mathbb{A}_L^{+}}
\newcommand{\AL}{\mathbb{A}_L}
\newcommand{\AEp}{\mathbb{A}_E^{+}}
\renewcommand{\AE}{\mathbb{A}_E}

\newcommand{\BK}{\mathbb{B}_K}

\newcommand{\BE}{\mathbb{B}_E}
\DeclareMathOperator{\Gal}{Gal}
\DeclareMathOperator{\Frac}{Frac}
\pgfplotsset{compat=1.18}
\begin{document}
\itemsep0pt
\title{A remark on an integral structure of the imperfect coefficient ring of $(\varphi,\Gamma)$-modules}
\author{Takumi Watanabe}
\address{Graduate School of Mathematical Sciences, The University of Tokyo, 3-8-1 Komaba, Meguro-ku, Tokyo, 153-8914, Japan}
\email{takumi0426@g.ecc.u-tokyo.ac.jp}
\maketitle
\begin{abstract}
Let $K$ be a complete discrete valuation field of characteristic $0$ with perfect residue field of characteristic $p>0$. Let $\AK$ denote the imperfect coefficient ring of $(\varphi,\Gamma)$-modules defined by Jean-Marc Fontaine. We prove that the canonical map $W(k_{K_\infty})[[\mu]]\rightarrow \AK\cap A_\inf$ is an isomorphism, even when $K$ is ramified. 
This fact was remarked by Nathalie Wach without proof. In Appendix 2, we include a result of Dylan Pentland. Both results indicate the difficulty of constructing a coefficient ring of ``Wach modules'' in the ramified case.
\end{abstract}
\tableofcontents
\section*{Introduction}

In \cite{Fontaine1990}, Jean-Marc Fontaine constructed a theory of $(\varphi,\Gamma)$-modules in order to study $p$-adic Galois representations. He constructed a subring $\AK \subseteq W(\widehat{K_\infty}^\flat)$ that is stable under the Frobenius endomorphism and the $\Gamma_K$-action, which was denoted by $\mathcal{O}_{\mathcal{E}}$ in \cite{Fontaine1990}. When $K$ is absolutely unramified, $\AK=\bigl(W(k)[[\mu]][1/\mu]\bigr)^\wedge_p$. A \emph{free \'{e}tale $(\varphi,\Gamma)$-module over $\AK$} is a finite free $\AK$-module $M$ equipped with a continuous semi-linear $\Gamma_K$-action and a $\Gamma_K$-equivariant $\AK$-linear isomorphism $M\otimes_{\AK, \varphi}\AK \cong M$. J.-M. Fontaine showed that the category of free \'{e}tale $(\varphi,\Gamma)$-modules over $\AK$ is equivalent to the category of free $\mathbb{Z}_p$-representations of $G_K$. Here, a \emph{free $\mathbb{Z}_p$-representation of $G_K$} is a finite free $\mathbb{Z}_p$-module equipped with a continuous linear $G_K$-action. As $\Gamma_K$ can be embedded into $\mathbb{Z}_p^\times$ by the cyclotomic character, the $\Gamma_K$-action is simpler than that of $G_K$. Therefore, the theory of $(\varphi,\Gamma)$-modules is very important and useful.

Among $p$-adic Galois representations, crystalline representations are one of the important classes. Thus, it is important to characterize the $(\varphi,\Gamma)$-modules corresponding to crystalline representations. After innovative works by J.-M. Fontaine \cite{Fontaine1990}, Nathalie Wach \cite{Wach1996,Wach1997}, and Pierre Colmez \cite{Colmez1999}, Laurent Berger constructed a theory of Wach modules in \cite{berger2004} when $K$ is absolutely unramified. A \emph{Wach module} is a finite free $\AKp:=W(k)[[\mu]]$-module $N$ equipped with a continuous semi-linear $\Gamma_K$-action that is trivial on $N/\mu$ and a $\Gamma_K$-equivariant $\AKp[1/\varphi(\xi)]$-linear isomorphism $N \otimes_{\AKp,\varphi}\AKp[1/\varphi(\xi)]\cong N[1/\varphi(\xi)]$. Here, $\varphi(\xi):= \varphi(\mu)/\mu \in \AKp$.
For any Wach module $N$, $N\otimes_{\AKp}\AK$ is the $(\varphi,\Gamma)$-module over $\AK$ corresponding to the crystalline $\mathbb{Z}_p$-representation. Hence, Wach modules can be regarded as a special lattice of $(\varphi,\Gamma)$-modules.
L. Berger showed that the category of Wach modules is equivalent to the category of free crystalline $\mathbb{Z}_p$-representations of $G_K$. Since then, Wach modules have been used not only in the study of $p$-adic Galois representations, but also, for example, in the study of Iwasawa theory and $p$-adic Langlands correspondence.

However, Wach modules can be used only when $K$ is absolutely unramified. This is because a good coefficient ring is unknown when $K$ is ramified. The coefficient ring should be a ``good integral subring'' of $\AK$ stable under the Frobenius endomorphism and the $\Gamma_K$-action. A good candidate might be $\AK\cap A_\inf$. When $K$ is absolutely unramified, it coincides with $\AKp=W(k)[[\mu]]$. However, N. Wach remarked in \cite[REMARQUE on page 393]{Wach1996} without proof that even if $K$ is ramified, we have an equality 
\begin{align*}
    W(k_{K_\infty})[[\mu]]= \AK \cap A_\inf \quad \text{in } W(\mathbb{C}_p^\flat),
\end{align*}
where $k_{K_\infty}$ is the residue field of $K_\infty$. Since the left-hand side does not have any information about the ramification of $K$, this fact implies that we cannot use $\AK\cap A_\inf$ as a coefficient ring of ``Wach modules'' when $K$ is ramified. We give a proof of the equality in this paper.

Another candidate for a good coefficient ring might be the $p$-adic completion of the integral closure of $\AFp$ in $\AK$. In Appendix 1, we explain why this does not work well.
In Appendix 2, we include the following result of Dylan Pentland. 

\begin{proposition}[{\cref{prop:Deylan Pentland}}]
Let $K= \mathbb{Q}_p(p^{1/e})$ with $e\geq 2$ and $\gcd(e, p(p-1))=1$. Then, there is no subring $R$ of $\mathbb{A}_K$ which is stable under the Frobenius endomorphism and with the property that $R/p \simeq \mathbb{E}_K^+ \subset \mathbb{E}_K$ via the natural map $R/p \rightarrow \mathbb{A}_K/p=\mathbb{E}_K$.
\end{proposition}

This result also indicates the difficulty of constructing a coefficient ring of ``Wach modules'' in the ramified case.

\subsection*{Acknowledgements}
This paper is based on the great ideas of Laurent Berger and Takeshi Tsuji. 
I would like to express my sincere gratitude to Laurent Berger for providing the basic ideas of the proof and for the useful comments on the draft. Especially, the idea of regarding $\AK \cap A_\inf$ as a kind of Wach module over $\AFp$ is due to him.
I am also deeply grateful to my advisor, Takeshi Tsuji, for giving me essential ideas of the proof and for his valuable comments on the draft. In particular, he provided the key ideas for proving \cref{prop:weaker main theorems (1)}.
I believe that both of them deserve to be authors of this paper. However, they politely declined. 
All mistakes in this paper are entirely my own. 
It was Abhinandan who first informed me of \cite[REMARQUE on page 393]{Wach1996}. I thank him for this and for giving me helpful comments on the draft.
I would also like to thank Dylan Pentland for allowing me to include his great discovery as an appendix and for the useful discussion on the contents of Appendix 1.
This work was supported by the WINGS-FMSP program at Graduate School of Mathematical Sciences, the University of Tokyo. This work was partially supported by JSPS KAKENHI Grant Number 25KJ1165.

\subsection*{Notation and Conventions}
Throughout this paper, we fix a complete discrete valuation field $F$ of characteristic $0$ with perfect residue field $k_F$ of characteristic $p>0$ that is absolutely unramified, i.e., $p$ is a uniformizer. We also fix an algebraic closure $\overline{F}$ of $F$ and let $\mathbb{C}_p$ denote the completion of $\overline{F}$. We fix a compatible system $\{\zeta_{p^n}\}_{n\geq 1}$ of primitive $p$-power roots of unity in $\overline{F}$.
In this paper, a \emph{$p$-adic field} means a finite extension of $F$ in $\overline{F}$, which is a complete discrete valuation field of characteristic $0$ with perfect residue field of characteristic $p$.
For any $p$-adic field $K$, we set $K_n:=K(\zeta_{p^n})$ for any positive integer $n\geq 1$ and $K_\infty:=\bigcup_{n\geq 1}K_n$. 
Let $G_K$ denote the absolute Galois group $\Gal \left(\overline{F}/K\right)$ of $K$. We also set $H_K:=\Gal \left(\overline{F}/K_\infty\right)$ and $\Gamma_K:=\Gal \left(K_\infty/K\right)$. For any subfield $L\subseteq \mathbb{C}_p$, $\mathcal{O}_L$ denotes the ring of integers and $k_L$ denotes the residue field of $L$. 
For any $\mathbb{F}_p$-algebra $R$, $W(R)$ denotes the ring of $p$-typical Witt vectors of $R$.
We set $A_\inf:=W(\mathcal{O}_{\mathbb{C}_p}^{\flat})$, where $\mathcal{O}_{\mathbb{C}_p}^{\flat}$ is the tilt $\varprojlim_{x\mapsto x^p}\mathcal{O}_{\mathbb{C}_p}/p$ of $\mathcal{O}_{\mathbb{C}_p}$. Let $\mathbb{C}_p^\flat$ denote the field of fractions of $\mathcal{O}_{\mathbb{C}_p}^\flat$. Similarly, we define $\mathcal{O}_{\widehat{K_\infty}}^\flat := \varprojlim_{x\mapsto x^p} \mathcal{O}_{ \widehat{K_\infty}}/p$ and $\widehat{K_\infty}^\flat:= \Frac \mathcal{O}_{\widehat{K_\infty}}^\flat$ for any $p$-adic field $K$. We put $\epsilon:=(1,\zeta_p,\zeta_{p^2},\dots )\in \mathcal{O}_{\widehat{F_\infty}}^{\flat}$ and $\mu:=[\epsilon]-1\in W(\mathcal{O}_{\widehat{F_\infty}}^\flat)$.

\section*{Proof of the Main Theorem}
For any $p$-adic field $K$, let $\AK$ denote the imperfect coefficient ring of $(\varphi,\Gamma)$-modules defined by J.-M. Fontaine in \cite{Fontaine1990}, which was denoted by $\mathcal{O}_{\mathcal{E}}$. For the sake of the reader, we recall the definition here.
We set $\mathbb{E}_F:= k_F((\epsilon-1))\subseteq \mathbb{C}_p^\flat$. It is a complete discrete valuation field of characteristic $p$. Since $\mathbb{C}_p^\flat$ is algebraically closed, we may consider the separable closure $\mathbb{E}$ of $\mathbb{E}_F$ inside $\mathbb{C}_p^\flat$. 
We see that $\mathbb{E}$ is stable under the canonical $G_F$-action on $\mathbb{C}_p^\flat$. Moreover, $H_F$ acts on $\mathbb{E}_F$ trivially, so we have a homomorphism $H_F \rightarrow \Gal \left(\mathbb{E}/\mathbb{E}_F\right)$. By the theorem of Fontaine-Wintenberger (\cite[Th\'{e}or\`{e}me 3.1.6]{Fontaine1990}, \cite{Wintenberger1983}), it is an isomorphism. For any $p$-adic field $K$, we define $\mathbb{E}_K:= \mathbb{E}^{H_K}$. Then, $\mathbb{E}_K$ is a complete discrete valuation field with residue field $k_{K_\infty}$, and it is a finite extension of $\mathbb{E}_F$ of degree $[K_\infty: F_\infty]$. We define $\mathbb{E}_K^+:=\mathbb{E}_K \cap \mathcal{O}_{\mathbb{C}_p}^\flat$. 
One can show that $\mathbb{E}_K^+ = k_{K_\infty}[[\widetilde{\pi}_{K}]]$, where $\widetilde{\pi}_{K}$ is a uniformizer of $\mathbb{E}_K$. 
For any finite unramified extension $K$ of $F$, we define $\AK$ as follows. Consider a $W(k_K)$-linear homomorphism $W(k_K)[T]\rightarrow W(\mathcal{O}_{\widehat{K_\infty}}^\flat)$ defined by $T\mapsto \mu$, where $W(k_K)[T]$ is the polynomial ring over $W(k_K)$ in one variable $T$. Since $\mu^n \rightarrow 0$ in $W(\mathcal{O}_{\widehat{K_\infty}}^\flat)$ as $n \rightarrow \infty$ with respect to the weak topology, we obtain a homomorphism $W(k_K)[[T]]\rightarrow W(\mathcal{O}_{\widehat{K_\infty}}^\flat)$ of $W(k_K)$-algebras. Since $\mu$ is invertible in $W(\widehat{K_\infty}^\flat)$ and $W(\widehat{K_\infty}^\flat)$ is $p$-adically complete and separated, we obtain a homomorphism 
\begin{align*}
    \bigl(W(k_K)[[T]][1/T]\bigr)^\wedge_p \rightarrow W(\widehat{K_\infty}^\flat).
\end{align*}
It is injective because its reduction modulo $p$ is equal to the canonical injection $k_K ((\epsilon-1))\hookrightarrow \widehat{K_\infty}^\flat$, and the source and the target rings are $p$-adically complete and separated and $p$-torsion free. We define $\AK$ as the image of this homomorphism.
The following lemma is well-known. 

\begin{lemma}\label{lem:main thm for unramified extension}
     For any finite unramified extension $K$ of $F$, we define $\AKp:=\AK \cap A_\inf$. Then, we have
     \begin{align*}
          \AKp = W(k_K)[[\mu]].
     \end{align*}
     Here, the right-hand side is defined as the image of the injective $W(k_K)$-algebra homomorphism $W(k_{K})[[T]] \lhook\joinrel\rightarrow A_\inf$, $T\mapsto \mu$, where the limit is taken with respect to the weak topology.
     \end{lemma}
\begin{proof}[Sketch of a proof]
     For the sake of the reader, we write a sketch of a proof here. 
     It is easy to see that $W(k_K)[[\mu]]\subseteq \AKp$.
     Suppose $W(k_K)[[\mu]] \subsetneq \AKp$. We have an equality  
     \begin{align*}
          \AK=\left\{\lim_{N\to \infty } \sum_{i=-N}^\infty a_i \mu^i\in W(\mathbb{C}_p^\flat)\middle|\right.\begin{split} a_i\in W(k_K), a_i \rightarrow 0\text{ as }  i \rightarrow -\infty \\
          \text{ with respect to the $p$-adic topology}\end{split}\Biggr\}.
     \end{align*}
     Here, the limit $\sum_{i=-N}^\infty$ is taken with respect to the weak topology and the limit $ \lim_{N\to \infty}$ is taken with respect to the $p$-adic topology. 
     Then, there exists a non-zero element 
     \begin{align*}
          \lim_{N\to \infty} \sum_{i=-N}^{-1}a_i \mu^i\in \AKp
     \end{align*}
     with $a_i\in W(k_K)$. Note that for every $x\in\AK$, $px\in\AKp$ implies $x\in \AKp$ as $W(\mathbb{C}_p^\flat)/A_\inf$ is $p$-torsion free, which can be proved by the injectivity of $\mathcal{O}_{\mathbb{C}_p}^\flat \rightarrow \mathbb{C}_p^\flat$, $p$-torsion freeness of $W(\mathbb{C}_{p}^\flat)$, and the snake lemma. Thus, we may assume that there exists $N\geq 1$ such that $a_N\notin pW(k_K)$. Let $N$ be the maximal integer such that $a_{-N}\notin pW(k_K)$. Then
     \begin{align*}
          \sum_{i=-N}^{-1}\overline{a_i} \widetilde{\pi}^i\in \mathcal{O}_{\mathbb{C}_p}^\flat
     \end{align*}
     where $\widetilde{\pi}:=\epsilon-1$ and $\overline{a_i}\in k_K$ is $a_i \bmod p$. We can deduce a contradiction by using the valuation.
\end{proof}

Let us consider a $p$-adic field $K$. Let $E$ be the maximal unramified extension in $K_\infty/F$. Then $E$ is a finite extension of $F$ and $E_\infty \subseteq K_\infty$. 
By the definition of $E$, $K_\infty/E_\infty$ is totally ramified, i.e., the residue field of $K_\infty$ coincides with that of $E_\infty$, which is also equal to that of $E$. Thus, $\mathbb{E}_K/\mathbb{E}_E$ is a totally ramified extension. 
We fix a uniformizer $\widetilde{\pi}_K\in \mathbb{E}_K^+$ of $\mathbb{E}_K$ and let $\overline{P}(T)\in \mathbb{E}_E^+[T]$ be the minimal polynomial of $\widetilde{\pi}_K$ over $\mathbb{E}_E$. Note that $\overline{P}(T)$ is an Eisenstein polynomial and $\AEp/p \cong \mathbb{E}_E^+$ by \cref{lem:main thm for unramified extension}. Let $P(T)\in \AEp[T]$ be any lift of $\overline{P}(T)\in \mathbb{E}_E^+[T]$. Since $\overline{P}(T)\in \mathbb{E}_E[T]$ is separable, $\overline{P}'(\widetilde{\pi}_K)\in \mathbb{E}_K^\times\subseteq \widehat{K_\infty}^{\flat,\times}$. By applying Hensel's lemma to $P(T)\in W(\widehat{K_\infty}^\flat)[T]$, we obtain a root $\pi_K\in W(\widehat{K_\infty}^\flat)$ of $P(T)$ such that $\pi_K\bmod p =\widetilde{\pi}_K$. Note that $\pi_K$ is algebraic over $\AE[1/p]$. We can construct a $W(k_{K_\infty})$-linear injective homomorphism
\begin{align*}
     \Bigl(W(k_{K_\infty})[[T]][1/T]\Bigr)^\wedge_p\lhook\joinrel\rightarrow W(\widehat{K_\infty}^\flat),\qquad T\mapsto \pi_K
\end{align*}
in the same manner as before. We define $\AK$ as the image of this homomorphism. Then, $\AK$ is characterized by the following properties:
\begin{itemize}
    \item $\AK$ is a subring of $W(\mathbb{C}_p^\flat)$.
    \item $\AK$ is a Cohen ring of $\mathbb{E}_K$ with an isomorphism $\AK/p \cong \mathbb{E}_K$ induced by $\AK \hookrightarrow W(\mathbb{C}_p^\flat)$.
    \item $\AK$ contains $\AF$ as a subring and the canonical homomorphism $\AF \rightarrow \AK$ is \'{e}tale.
\end{itemize}
From this characterization, we see that $\AK$ is a subring of $W(\widehat{K_\infty}^\flat)$ stable under the Frobenius endomorphism and the $\Gal \left(K_\infty/F\right)$-action. We define $\mathbb{A}$ as the $p$-adic completion of $\bigcup_{K/F \text{:finite}} \mathbb{A}_K \subseteq W(\mathbb{C}_p^\flat)$, which is a subring of $W(\mathbb{C}_p^\flat)$ stable under the Frobenius endomorphism and the $G_F$-action. We set $\mathbb{B}_K:= \mathbb{A}_K[1/p]$, $\mathbb{B}:=\mathbb{A}[1/p]$. It is known that $\mathbb{A}_K=\mathbb{A}^{H_K}$ and $\mathbb{B}_K=\mathbb{B}^{H_K}$.

\begin{remark}\label{rem:definition of A_K}
The ring $\AK$ is independent of $F$ and it depends only on $\overline{F}$, which is an algebraic closure of $K$, in the following sense: For another absolutely unramified complete discrete valuation field $F'$ with perfect residue field in $\overline{F}$ such that $K$ is a finite extension of $F'$, the ring $\AK$ defined using $F'$ coincides with the ring $\AK$ defined using $F$. This can be shown by using the characterization of $\AK$ mentioned above. 
\end{remark}

\begin{definition}
For any $p$-adic field $K$, we define 
\begin{align*}
     \AKp:=\AK\cap A_\inf
\end{align*}
in $W(\mathbb{C}_p^\flat)$.
\end{definition}

The next proposition follows immediately from the definition, but it is important.

\begin{proposition}\label{prop:subextension of p-cyclotomic extension}
Let $K$ be a $p$-adic field and let $L$ be a finite extension in $K_\infty/K$. Then we have $\AK=\AL$ and $\AKp=\ALp$.
\end{proposition}

The following is the main theorem in this paper.

\begin{theorem}\label{thm:main theorem}
For any $p$-adic field $K$, we have
\begin{align*}
    \AKp=W(k_{K_\infty})[[\mu]].
\end{align*}
\end{theorem}
\begin{remark}
\cite[REMARQUE on page 393]{Wach1996} contains a subtle mistake. The right-hand side is not $W(k_{K})[[\mu]]$ since $\AKp$ contains $W(k_{K_\infty})$. Note that if $K$ is absolutely unramified, we have $k_{K_\infty}=k_K$.
\end{remark}

\begin{proposition}\label{prop:weaker main theorems}
    The following hold.
    \begin{enumerate}
        \item Let $K$ be a finite Galois extension of $F$ such that $K_\infty/F_\infty$ is totally ramified, i.e., the residue field of $K_\infty$ coincides with that of $F_\infty$. Then, we have $\AFp =\AKp$.
        \item Let $K$ be a finite (not necessarily Galois) extension of $F$ such that $K_\infty/F_\infty$ is totally ramified. Then, we have $\AFp =\AKp$.
    \end{enumerate}
\end{proposition} 

\begin{proposition}
The following hold.
\begin{enumerate}
    \item \cref{prop:weaker main theorems} (1) implies \cref{prop:weaker main theorems} (2).
    \item \cref{prop:weaker main theorems} (2) implies \cref{thm:main theorem}.
\end{enumerate}
\end{proposition}
\begin{proof}
(1)
Let $K$ be a finite (not necessarily Galois) extension of $F$ such that $K_\infty/F_\infty$ is totally ramified. Let $L$ be the Galois closure of $K/F$ and $M:=W(k_{L_\infty})[1/p]$.
Then, there exists a positive integer $n\geq 1$ such that the residue field of $L_n$ coincides with that of $L_\infty$. Then $M$ is contained in $L_n$. Note that $L_n$ is the composite field of $F_n $ and $L$, thereby $L_n/F$ is a finite Galois extension.
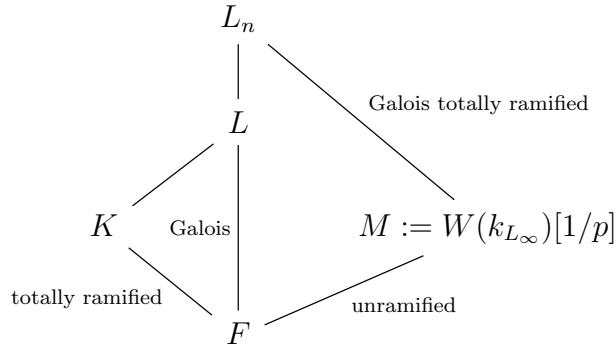
\[\begin{tikzcd}
&L_n& \\ 
&L \arrow[u, dash]& \\ 
K\arrow[ur, dash]&&M:=W(k_{L_\infty})[1/p]\arrow[luu,dash,"\text{Galois totally ramified}"'] \\ 
&F\arrow[uu, dash, "\text{Galois}"]\arrow[lu, "\text{totally ramified}", dash]\arrow[ru, "\text{unramified}"',dash]&
\end{tikzcd}\]
By the definition of $M$, $(L_n)_\infty/M_\infty$ is totally ramified. Thus, we can apply \cref{prop:weaker main theorems} (1) to $L_n/M$ and we see that $\mathbb{A}_M^+=\mathbb{A}_{L_n}^+=\mathbb{A}_L^+$ by \cref{prop:subextension of p-cyclotomic extension}. 

As $\AKp \subseteq \ALp$, it suffices to show that $\mathbb{A}_M^+\cap\AKp=\AFp$, which is equivalent to showing that
\begin{alignat*}{2}
    &&\mathbb{A}_M\cap \AK&=\AF\\
    \Leftrightarrow \quad&& \mathbb{A}^{H_M}\cap \mathbb{A}^{H_K}&=\mathbb{A}^{H_F}\\
    \Leftrightarrow \quad&&\mathbb{A}^{H_MH_K}&=\mathbb{A}^{H_F}\\
    \Leftarrow\quad&& H_MH_K&=H_F\\
    \Leftrightarrow\quad&&M_\infty\cap K_\infty&=F_\infty.
\end{alignat*}
Here, $H_MH_K$ denotes the group generated by $H_M$ and $H_K$. Note that we used the fact that it is closed in $\Gal \left(\overline{F}/F\right)$, since $H_M$ and $H_K$ are open subgroups of $\Gal \left(\overline{F}/F_\infty\right)$. Hence, it suffices to show that $M_n\cap K_n = F_n$ for any positive integer $n\geq 1$. By assumption, $K_n/F_n$ is totally ramified. Since $F_n/F$ is totally ramified and $M/F$ is unramified, we see that $M_n/F_n$ is unramified. Hence we conclude that $M_n\cap K_n = F_n$.
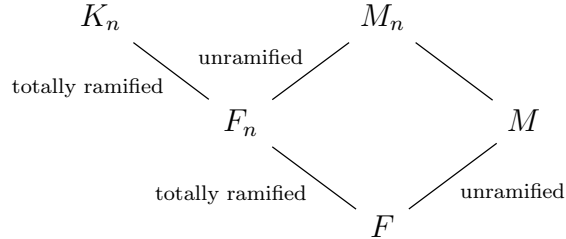
\[\begin{tikzcd}
K_n&&M_n& \\ 
&F_n\arrow[ru, "\text{unramified}",dash]\arrow[ul, "\text{totally ramified}",dash]&&M\arrow[lu, dash] \\ 
&&F\arrow[ru, dash,"\text{unramified}"']\arrow[lu, dash,"\text{totally ramified}"]&
\end{tikzcd}\]

(2)
Let $K$ be an arbitrary $p$-adic field. Then, there exists a positive integer $n\geq 1$ such that the residue field of $K_n$ coincides with that of $K_\infty$. Let $F':=W(k_{K_\infty})[1/p]$. Then $F'$ is contained in $K_n$ and $(K_{n})_\infty/F'_\infty$ is totally ramified. \cref{lem:main thm for unramified extension,prop:subextension of p-cyclotomic extension,prop:weaker main theorems} (2) imply $\AKp=\mathbb{A}_{K_n}^+=\mathbb{A}_{F'}^+=W(k_{K_\infty})[[\mu]]$ as desired.
\end{proof}

It remains to show \cref{prop:weaker main theorems} (1). We prove it by using \cite[TH\'{E}OR\`{E}ME, REMARQUES on page 381]{Wach1996}. 
We first study some properties of $\AKp$. Since $\AK/p=\mathbb{E}_K \subseteq \mathbb{C}_p^\flat$ and $A_\inf/p=\mathcal{O}_{\mathbb{C}_p}^\flat\subseteq \mathbb{C}_p^\flat$, we can consider their intersection in $\mathbb{C}_p^\flat$, which coincides with $k_{K_\infty}[[\widetilde{\pi}_K]]$ where $\widetilde{\pi}_K\in \mathbb{E}_K$ is a uniformizer of $\mathbb{E}_K$.

\begin{lemma}\label{lem:mod p injection of AKp}
For any $p$-adic field $K$, the canonical homomorphism
\begin{align*}
     \AKp/p  \rightarrow \AK/p \cap A_\inf/p=k_{K_\infty}[[\widetilde{\pi}_K]]=:\mathbb{E}_K^+
\end{align*}
is injective.
\end{lemma}
\begin{proof}
It is enough to show that the canonical homomorphism
\begin{align*}
     \AKp/p  \rightarrow W(\mathbb{C}_p^\flat)/p
\end{align*}
is injective. Let $a\in \AKp$ such that $a\in pW(\mathbb{C}_p^\flat)$. Since the canonical homomorphism $A_\inf/p \lhook\joinrel\rightarrow W(\mathbb{C}_p^\flat)/p$ is injective, we have $a\in p A_\inf$. Since the canonical homomorphism $\AK/p \lhook\joinrel\rightarrow W(\mathbb{C}_p^\flat)/p$ is injective, we have $a\in p \AK$. As $W(\mathbb{C}_p^\flat)$ is $p$-torsion free and $\AKp:=\AK \cap A_\inf$, we see that $a\in p\AKp$ as desired.
\end{proof}

\begin{proposition}\label{prop:p-adically complete}
For any $p$-adic field $K$, $\AKp$ is $p$-adically complete and separated.
\end{proposition}
\begin{proof}
Consider the following commutative diagram:
\[\begin{tikzcd}
&\AK\arrow[rrr, "\cong "]\arrow[rd,hook]&&& \varprojlim_{n}\AK/p^n\arrow[rd]\arrow[from=ldd]& \\ 
&&W(\mathbb{C}_p^\flat)\arrow[rrr, "\cong ", pos=0.4,crossing over]&&& \varprojlim_{n}W(\mathbb{C}_p^\flat)/p^n \\ 
\AKp\arrow[ruu, hook]\arrow[rrr]\arrow[rd, hook]&&& \varprojlim_{n}\AKp/p^n\arrow[rd]&& \\ 
&A_\inf\arrow[ruu, hook,crossing over]\arrow[rrr, "\cong "]&&& \varprojlim_{n}A_\inf/p^n\arrow[ruu]&.
\end{tikzcd}\]
It is easy to show that $\AKp  \rightarrow  \varprojlim_{n}\AKp/p^n$ is injective, so it suffices to show that it is surjective. Let $(a_n)_{n\geq  1} \in  \varprojlim_{n}\AKp/p^n$ be an arbitrary element. Let $(b_n)_{n\geq 1}\in  \varprojlim_{n}\AK/p^n$ (resp. $(c_n)_{n\geq 1}\in  \varprojlim_{n}A_\inf/p^n$) be the image of $(a_n)_{n\geq 1}$. 
Since $\AK$ (resp. $A_\inf$) is $p$-adically complete and separated, there exists $b\in\AK$ (resp. $c\in A_\inf$) whose image in $ \varprojlim_{n}\AK/p^n$ (resp. $ \varprojlim_{n}A_\inf/p^n$) is $(b_n)_{n\geq 1}$ (resp. $(c_n)_{n\geq 1}$). Since the image of $(b_n)_{n\geq 1}$ in $ \varprojlim_{n}W(\mathbb{C}_p^\flat)/p^n$ coincides with that of $(c_n)_{n\geq 1}$ in $ \varprojlim_{n}W(\mathbb{C}_p^\flat)/p^n$ and $W(\mathbb{C}_p^\flat)$ is $p$-adically complete and separated, $b=c$ in $W(\mathbb{C}_p^\flat)$. Hence $b=c=:a\in \AKp$. 
It remains to show that the image of $a$ in $ \varprojlim_{n}\AKp/p^n$ is $(a_n)_{n\geq 1}$. It suffices to show that $ \varprojlim_{n}\AKp/p^n  \rightarrow  \varprojlim_{n}W(\mathbb{C}_p^\flat)/p^n$ is injective. Since $\varprojlim_{n}W(\mathbb{C}_p^\flat)/p^n\cong W(\mathbb{C}_p^\flat)$ is $p$-torsion free and $ \varprojlim_{n}\AKp/p^n$ is $p$-adically separated, it is enough to show that $\AKp/p  \rightarrow W(\mathbb{C}_p^\flat)/p$ is injective, which we have already shown in the proof of \cref{lem:mod p injection of AKp}.
\end{proof}

\begin{proposition}\label{prop:finite freenness of AKp}
    For any $p$-adic field $K$, the $\AFp$-module $\AKp$ is finite free.
\end{proposition}
\begin{proof}
Since $\AFp$ and $\AKp$ are $p$-adically complete and $p$-torsion free, it is enough to show that $\AKp/p$ is a finite free $\AFp/p$-module. By \cref{lem:mod p injection of AKp}, we have an injection
\begin{align*}
     \AKp/p \lhook\joinrel\rightarrow k_{K_\infty}[[\widetilde{\pi}_K]].
\end{align*}
Since $k_{K_\infty}[[\widetilde{\pi}_K]]$ is a torsion free finitely generated $\AFp/p$-module, $\AKp/p$ is a torsion free finitely generated $\AFp/p$-module. As $\AFp/p$ is a complete discrete valuation ring, this implies that $\AKp/p$ is a finite free $\AFp/p$-module, as desired.
\end{proof}



\begin{lemma}\label{lem:AKp/mu injets A_inf/mu}
For any $p$-adic field $K$, the canonical homomorphism $\AKp/\mu  \rightarrow A_\inf /\mu$ is injective.
\end{lemma}
\begin{proof}
It suffices to show $\AKp\cap \mu A_\inf \subseteq \mu\AKp$. Let $a\in\AKp$ such that $a=\mu x$ for some $x\in A_\inf$. Then $x=\mu^{-1}a$ in $W(\mathbb{C}_p^\flat)$. Since $\mu$ is invertible in $\AK$, this implies $x\in \AK$. Thus, $x\in \AK\cap A_\inf =:\AKp$ and $a\in \mu\AKp$.
\end{proof}

The next proposition and its proof are due to L. Berger.

\begin{proposition}[L. Berger]\label{prop:Gamma_F_n action is trivial mod mu}
    For any $p$-adic field $K$, the $\Gamma_{K_n}$-action on $\AKp/\mu$ is trivial for $n\gg 1$.
\end{proposition}
\begin{proof}
Consider the homomorphism  
\begin{align*}
     f\colon \AKp/\mu  \rightarrow A_\inf/\mu  \rightarrow \prod_{i\geq 0}\mathcal{O}_{\mathbb{C}_p},\qquad a\mapsto a\mapsto (\theta\circ \varphi^i(a))_{i\geq 0}
\end{align*}
where $\theta\colon A_\inf  \rightarrow \mathcal{O}_{\mathbb{C}_p}$ is the Fontaine's map. 
By \cite[5.1.4. Lemme.]{Fontaine1994} and \cref{lem:AKp/mu injets A_inf/mu}, $f$ is injective. Since $\theta(\AFp)=W(k_F)$ and $\AKp$ is a finite free $\AFp$-module by \cref{prop:finite freenness of AKp}, $\theta(\AKp)$ is a finitely generated $W(k_F)$-module. 
Moreover, $\theta(\AKp)$ is a subring of $\theta(W(\mathcal{O}_{\widehat{K_\infty}}^\flat))=\mathcal{O}_{\widehat{K_\infty}}$.  As $\overline{K}\cap \widehat{K_\infty}\subseteq \overline{K}^{H_K}=K_\infty$, we conclude that $\theta(\AKp)\subseteq \mathcal{O}_{K_n}$ for $n\gg 1$. Therefore, we obtain an injection
\begin{align*}
     f\colon \AKp/\mu \lhook\joinrel\rightarrow \prod_{i\geq 0}\mathcal{O}_{K_n},\qquad a\mapsto (\theta \circ \varphi^i(a))_{i\geq 0}.
\end{align*}
This homomorphism is $\Gamma_{K_n}$-equivariant, so we get the conclusion.
\end{proof}

\begin{proposition}[cf. {\cite[PROPOSITION 3.1.]{Liu2008}}]\label{prop:(phi Gamma) module corres to regular rep}
Let $L/K$ be a finite Galois extension of $p$-adic fields such that $K_\infty\cap L=K$. Then the $(\varphi,\Gamma)$-module over $\AK$ corresponding to the $\mathbb{Z}_p$-representation $\mathbb{Z}_p[\Gal \left(L/K\right)]$ of $G_K$ is $\AL$ equipped with the canonical Frobenius endomorphism and the $\Gamma_K$-action defined via the isomorphism $\Gamma_L  \xrightarrow{\cong } \Gamma_K$.
\end{proposition}
\begin{proof}
We put $H_{L/K}:=\Gal \left(L_\infty/K_\infty\right)$.
\[\begin{tikzcd}
&L_\infty& \\ 
L\arrow[ru, dash,"\Gamma_L"]&&K_\infty \arrow[lu, dash, "H_{L/K}"']\\ 
&K \arrow[ru, dash,"\Gamma_K"']\arrow[lu, dash,"\Gal \left(L/K\right)"]&
\end{tikzcd}\]
By assumption, we have isomorphisms $H_{L/K}\xrightarrow{\cong}  \Gal \left(L/K\right)$ and $\Gamma_L \xrightarrow{\cong } \Gamma_K$. We identify them through these isomorphisms.
Let $D(-)$ denote the functor from the category of free $\mathbb{Z}_p$-representations of $G_K$ to the category of free \'{e}tale $(\varphi,\Gamma)$-modules over $\AK$ defined in \cite[A 3.4.2]{Fontaine1990}.
We have
\begin{align*}
     D(\mathbb{Z}_p[\Gal \left(L/K\right)])=\bigl(\mathbb{Z}_p[\Gal \left(L/K\right)]\otimes_{\mathbb{Z}_p}\mathbb{A}\bigr)^{H_K}=\bigl(\AL[\Gal \left(L/K\right)]\bigr)^{H_{L/K}}.
\end{align*}
On the one hand, for any $\sum_{g\in H_{L/K}}a_g\cdot g\in \AL[\Gal(L/K)]$, where $a_g\in \AL$ for each $g\in H_{L/K}$, we have
\begin{alignat*}{3}
     &&h\left(\sum_{g\in H_{L/K}}a_g\cdot g\right)&=\sum_{g\in H_{L/K}}a_g\cdot g&&\quad\text{for any $h\in H_{L/K}$}\\
      \Leftrightarrow \quad&& \sum_{g\in H_{L/K}} h(a_g)\cdot hg &=\sum_{g\in H_{L/K}}a_g\cdot g&&\quad\text{for any $h\in H_{L/K}$}\\
       \Leftrightarrow\quad&& h(a_g)&=a_{hg} &&\quad\text{for any $h,g\in H_{L/K}$}\\
        \Leftrightarrow\quad&& a_g&=g(a_1) &&\quad\text{for any $g\in H_{L/K}$}.
\end{alignat*}
Thus, every element of $\bigl(\AL[\Gal \left(L/K\right)]\bigr)^{H_{L/K}}$ can be uniquely expressed by $\sum_{g\in H_{L/K}}g(a)\cdot g$ for some $a\in\AL$. On the other hand, every element of the form $\sum_{g\in H_{L/K}}g(a)\cdot g$ for some $a\in\AL$ in $\AL [\Gal(L/K)]$ is fixed by $H_{L/K}$. Hence we obtain a bijective $\AK$-linear map
\begin{align*}
    f\colon \AL  \rightarrow \bigl(\AL[\Gal \left(L/K\right)]\bigr)^{H_{L/K}},\qquad a\mapsto \sum_{g\in H_{L/K}} g(a)\cdot g.
\end{align*}

It suffices to show that $f$ is compatible with the Frobenii and the $\Gamma_K$-actions. We claim that the Frobenius map and the $\Gamma_K$-action on $\bigl(\AL[\Gal \left(L/K\right)]\bigr)^{H_{L/K}}$ are given by
\begin{align*}
     \varphi\left(\sum_{g\in H_{L/K}}g(a)\cdot g\right)=\sum_{g\in H_{L/K}} \varphi(g(a))\cdot g\\
     \gamma\left(\sum_{g\in H_{L/K}}g(a)\cdot g\right)=\sum_{g\in H_{L/K}} \gamma g(a)\cdot  g
\end{align*}
for any $\gamma\in \Gamma_K$. Here $\Gamma_K$ acts on $\AL$ via the isomorphism $\Gamma_L \xrightarrow{\cong } \Gamma_K$. The claim for the Frobenius map is clear by definition. For any $\gamma\in \Gamma_K$, let $\gamma'\in \Gamma_L$ be the corresponding element through $\Gamma_L \xrightarrow{\cong} \Gamma_K$. Then, $\gamma'\in \Gamma_L \subseteq \Gal \left(L_\infty/K\right)$ is a lift of $\gamma$ under $\Gal(L_\infty/K)\twoheadrightarrow \Gamma_K$. We see that the $\gamma$-action is given by the canonical $\gamma'$-action on $\AL$ and the $\gamma'$-action on $\Gal (L/K)$ through $\Gal(L_\infty/K)\twoheadrightarrow \Gal(L/K)$, which is trivial. Hence, we get the desired formula.

Since the Frobenius map on $\AL$ is $H_{L/K}$-equivariant, $f$ is compatible with the Frobenii. We have $\Gamma_L\times H_{L/K}\xrightarrow{\cong} \Gal \left(L_\infty/K\right)$. Therefore, for any $\gamma\in \Gamma_K$, we have
\begin{align*}
     \gamma f(a)=\gamma\left(\sum_{g\in H_{L/K}}g(a)\cdot g\right)&=\sum_{g\in H_{L/K}} \gamma g(a)\cdot  g
     =\sum_{g\in H_{L/K}} g \gamma(a)\cdot g
     = f( \gamma(a)).
\end{align*}
\end{proof}

\begin{lemma}\label{lem:injection of (phi Gamma) module}
For any $p$-adic field $K$, the canonical homomorphisms
\begin{align*}
     \AKp/p \otimes_{\AFp/p}\mathbb{E}_F  \rightarrow \mathbb{E}_K\\
     \AKp \otimes_{\AFp} \AF  \rightarrow \AK
\end{align*}
are injective.
\end{lemma}
\begin{proof}
Since $\AKp \otimes_{\AFp} \AF$ and $\AK$ are $p$-torsion free and $p$-adically complete and separated by \cref{prop:finite freenness of AKp}, it suffices to show the injectivity of the first homomorphism. By \cref{lem:mod p injection of AKp}, we have the injection $\AKp/p \lhook\joinrel\rightarrow \mathbb{E}_K$. Note that $\mathbb{E}_F=\AFp/p[1/(\epsilon-1)]$ by \cref{lem:main thm for unramified extension}. Therefore, we obtain an injection
\begin{align*}
     \AKp/p \otimes_{\AFp/p}\mathbb{E}_F\lhook\joinrel\rightarrow \mathbb{E}_K \otimes_{\AFp/p}\mathbb{E}_F=\mathbb{E}_K
\end{align*}
as desired.
\end{proof}

Before proving \cref{prop:weaker main theorems} (1), let us recall \cite[TH\'{E}OR\`{E}ME, REMARQUES on page 381]{Wach1996}.

\begin{definition}[{\cite[A.5.]{Wach1996}}]\label{def:Wach A5}
Let $K$ be a totally ramified finite extension of $F$. For any $(\varphi,\Gamma)$-module $M$ over $\BK$, $j_\ast M$ denotes the union of finitely generated sub-$\AFp$-modules of $M$ stable under the Frobenius map.
\end{definition}

The canonical homomorphism $j_\ast M  \otimes_{\AFp[1/p]} \BK \rightarrow M$ is always injective.

\begin{definition}[{\cite[A.5.]{Wach1996}}]\label{def:finite height}
     Let $K$ be a totally ramified finite extension of $F$. A $(\varphi,\Gamma)$-module $M$ over $\BK$ is {\em of finite height} if the canonical injective homomorphism
     \begin{align*}
          j_\ast M  \otimes_{\AFp[1/p]} \BK \lhook\joinrel\rightarrow  M
     \end{align*}
     is an isomorphism.
\end{definition}

\begin{theorem}[{\cite[TH\'{E}OR\`{E}ME, REMARQUES on page 381]{Wach1996}}]\label{thm:Wach}
Let $K$ be a totally ramified finite extension of $F$ contained in $F_\infty$. Let $V$ be a $p$-adic Galois representation of $G_K$ of rank $d$ and let $M$ be the corresponding $(\varphi,\Gamma)$-module over $\mathbb{B}_K$. Assume that $M$ is of finite height. Then, the following are equivalent.
\begin{enumerate}
     \item $V$ is potentially crystalline.
     \item[(4)] There exists an integer $r$ and a sub-$\AFp$-module $N$ of $j_\ast M$ such that $N$ is free of rank $d$, stable under the $\Gamma_K$-action, and the $\Gamma_K$-action is finite on $(N/\mu)(-r)$.
\end{enumerate}
Moreover, if the $\Gamma_K$-action on $(N/\mu)(-r)$ is trivial, then $V$ is crystalline.
\end{theorem}

We finally prove \cref{prop:weaker main theorems} (1). The proof is based on the ideas of L. Berger and T. Tsuji.

\begin{proposition}[$=$ \cref{prop:weaker main theorems} (1)]\label{prop:weaker main theorems (1)}
Let $K$ be a finite Galois extension of $F$ such that $K_\infty/F_\infty$ is totally ramified, i.e., the residue field of $K_\infty$ coincides with that of $F_\infty$. Then, we have $\AFp =\AKp$.
\end{proposition}
\begin{proof}
Let $E:=F_\infty\cap K$. By \cref{prop:Gamma_F_n action is trivial mod mu}, there exists $n\geq 1$ such that the $\Gamma_{K_n}$-action on $\AKp/\mu$ is trivial and $E\subseteq F_n$. 
\[\begin{tikzcd}
&K_\infty \\ 
F_\infty\arrow[ur, dash]&K_n\arrow[u, dash] \\ 
F_n\arrow[u, dash]\arrow[ur, dash]&K\arrow[u, dash] \\ 
E=F_\infty\cap K\arrow[u, dash]\arrow[ur, dash]& \\ 
F \arrow[u, dash]&
\end{tikzcd}\]
Since $\Gal \left(K_\infty/K\right) \xrightarrow{\cong }\Gal \left(F_\infty/E\right)$ and $\Gal \left(K_n/K\right) \xrightarrow{\cong } \Gal \left(F_n/E\right) $, we have $\Gal \left(K_\infty/K_n\right)\cong \Gal \left(F_\infty/F_n\right)$. Hence we see that $F_\infty \cap K_n=F_n$. 
We know that $K_n$ is a finite Galois extension of $F$ and $\AKp=\mathbb{A}_{K_n}^+$ by \cref{prop:subextension of p-cyclotomic extension}. Thus we may replace $K$ by $K_n$. 
Then $E$ is replaced by $F_n$. In summary, we are reduced to the following settings: $K$ is a finite Galois extension of $F$ such that $K_\infty/F_\infty$ is totally ramified. 
We put $E:=F_\infty \cap K$. The $\Gamma_K$-action on $\AKp/\mu$ is trivial. Note that $\AF=\AE$, $\AFp=\AEp$ and $\mathbb{E}_F=\mathbb{E}_E$ by \cref{prop:subextension of p-cyclotomic extension}.

Consider $M:=\AKp\otimes_{\AEp} \AE$. We claim that $M$ is a $(\varphi,\Gamma)$-module over $\AE$. By \cref{prop:finite freenness of AKp}, $M$ is a finite free $\AE$-module and it has a canonical Frobenius map and a $\Gamma_E$-action through the isomorphism $\Gamma_K\xrightarrow{\cong} \Gamma_E$. By \cref{lem:injection of (phi Gamma) module}, we have an injection $M\lhook\joinrel\rightarrow \AK$ compatible with the Frobenii and the $\Gamma_E$-action. We also have the following commutative diagram induced by the Frobenii:
\[\begin{tikzcd}
\bigl(\AKp/p \otimes_{\AEp/p}\mathbb{E}_E\bigr)\otimes_{\mathbb{E}_E,\varphi}\mathbb{E}_E\arrow[r]\arrow[d,hook]&\AKp/p \otimes_{\AEp/p}\mathbb{E}_E\arrow[d, hook] \\ 
\mathbb{E}_K\otimes_{\mathbb{E}_E,\varphi}\mathbb{E}_E\arrow[r]&\mathbb{E}_K.
\end{tikzcd}\]
By \cref{prop:(phi Gamma) module corres to regular rep}, $\AK$ is a $(\varphi,\Gamma)$-module over $\AE$, thereby the bottom arrow is an isomorphism. Hence the top arrow is injective. As $\dim_{\mathbb{E}_E}\bigl(\AKp/p \otimes_{\AEp/p}\mathbb{E}_E\bigr)\otimes_{\mathbb{E}_E,\varphi}\mathbb{E}_E=\dim_{\mathbb{E}_E}\AKp/p \otimes_{\AEp/p}\mathbb{E}_E$, it is an isomorphism. 
Since $M$ is $p$-torsion free and $p$-adically complete and separated, this implies that the Frobenius map $M \otimes_{\AE, \varphi}\AE \rightarrow M$ is an isomorphism.

Before proving the continuity of the $\Gamma_E$-action on $M$, let us recall the definition of the topologies on these rings and modules. The topology on $\AEp$ (resp. $\AE$) is defined as the subspace topology induced from $A_\inf$ (resp. $W(\mathbb{C}_p^\flat)$) that is equipped with the weak topology. The topology on $\AEp$ is equal to the subspace topology induced from $\AE$. 
We claim that the topology on $\AEp$ is equal to the $(p,\mu)$-adic topology. It is known that the weak topology on $A_\inf$ coincides with the $(p, \mu)$-adic topology, hence it suffices to prove that $\AEp/(p^m,\mu^n)\rightarrow A_\inf/(p^m,\mu^n)$ is injective for any $n,m\in \mathbb{Z}_{\geq 1}$. By \cite[Lemma 2.10]{Watanabe2026}, $A_\inf/\mu^n$ is $p$-torsion free for any $n\in \mathbb{Z}_{\geq 1}$. Hence, it is enough to show the injectivity of $\AEp/(p,\mu^n)\rightarrow A_\inf/(p,\mu^n)$. When $n=1$, this is trivial since $\AEp/(p,\mu)$ is a field whose image in $A_\inf/(p,\mu)$ is nonzero. As $A_\inf/p$ is $\mu$-torsion free, this implies the injectivity for any $n\in \mathbb{Z}_{\geq 1}$, as desired. Therefore, the topology on $\AEp$ coincides with the $(p,\mu)$-adic topology. 
The finite free $\AE$-module $M$ has the product topology. Hence, by \cite[Proposition 1.11]{Watanabe2026}, it is enough to show the continuity of the $\Gamma_E$-action on $\AKp$ equipped with the product topology as a finite free $\AEp$-module, which coincides with the $(p,\mu)$-adic topology. Since the $\Gamma_E$-action on $\AKp$ is trivial modulo $\mu$, the desired continuity follows from \cite[Lemma 1.3]{MorrowTsuji2024} and \cite[Lemma 1.12]{Watanabe2026}. Therefore, $M$ is a $(\varphi,\Gamma)$-module over $\AE$.

We define $M_{\mathbb{Q}_p}:=M\otimes_{\AE}\BE$, which is an \'{e}tale $(\varphi,\Gamma)$-module over $\BE$. We next show that the $p$-adic Galois representation $V(M_{\mathbb{Q}_p})$ of $G_E$ corresponding to $M_{\mathbb{Q}_p}$ is the trivial representation $\mathbb{Q}_p$. By \cref{prop:finite freenness of AKp}, the canonical map $\AKp  \rightarrow M$ is injective. Since $j_\ast M_{\mathbb{Q}_p}$ contains $\AKp$, $M_{\mathbb{Q}_p}$ is of finite height. As the $\Gamma_E$-action on $\AKp/\mu$ is trivial, $V(M_{\mathbb{Q}_p})$ is a crystalline representation of $G_E$ by \cref{thm:Wach}.
Since $M_{\mathbb{Q}_p}\lhook\joinrel\rightarrow \BK$ as $(\varphi,\Gamma)$-modules over $\BE$, we see that 
\begin{align*}
     V(M_{\mathbb{Q}_p})\lhook\joinrel\rightarrow \mathbb{Q}_p[\Gal \left(K/E\right)]
\end{align*}
by \cref{prop:(phi Gamma) module corres to regular rep}.
Hence, $V(M_{\mathbb{Q}_p})$ is unramified by \cite[Cor. 9.3.2]{BrinonConrad2009}.  Let $I_E$ denote the inertia group of $E$. As $K/E$ is totally ramified, $I_E \lhook\joinrel\rightarrow G_E \twoheadrightarrow \Gal \left(K/E\right)$ is surjective. Thus,
\begin{align*}
     V(M_{\mathbb{Q}_p})\lhook\joinrel\rightarrow \bigl(\mathbb{Q}_p[\Gal \left(K/E\right)]\bigr)^{I_E}=\mathbb{Q}_p.
\end{align*}
We also have $\BE \lhook\joinrel\rightarrow M_{\mathbb{Q}_p}$ as $\AEp  \rightarrow \BE$ is flat, so we have an injection $\mathbb{Q}_p \lhook\joinrel\rightarrow V(M_{\mathbb{Q}_p})$. Therefore, $\BE \lhook\joinrel\rightarrow M_{\mathbb{Q}_p}$ is an isomorphism and hence $\AKp$ is a finite free $\AEp$-module of rank $1$. There exists $u\in \AKp$ such that $\AKp=\AEp\cdot u$. Then, there exists $v\in \AEp$ such that $1=vu$. This implies $\AEp = \AKp$.
\end{proof}

\section*{Appendix 1: The integral closure of $\AFp$ in $\AK$}
Another candidate for a coefficient ring of ``Wach modules'' in the ramified case might be the $p$-adic completion of the integral closure of $\AFp$ in $\AK$, as considered in \cite{Wu2021}. We explain why it does not work well, even in the unramified case. Roughly speaking, this is because the integral closure of $\AFp$ in $\AK$ is very large by virtue of Hensel's lemma.

Let $C$ be the integral closure of $\AFp$ in $\AK$ and $\widehat{C}$ be its $p$-adic completion. Note that $C$ is a subring of $\AK$, but a priori there is no reason for $\widehat{C}$ to be a subring of $\AK$. We claim that the canonical map $C/pC\cong\widehat{C}/p \widehat{C}\rightarrow \AK/p\AK= \mathbb{E}_K$ is not injective, in particular $\widehat{C}/p\widehat{C}$ is not isomorphic to $\mathbb{E}_K^+$. (Thus, \cite[Lemma 2.5]{Wu2021} does not hold.) The following argument was independently found by the author and Dylan Pentland.

Take a $1$-variable polynomial $f(X) = X^2 - \mu X + p$ in $\AFp[X]$. Modulo $p$, we see that $0$ is a simple root in $\mathbb{E}_F$. Using Hensel's lemma on the complete discrete valuation ring $\AF$, there is an element $x\in \mathbb{A}_F$ satisfying $f(x)=0$ and $x \equiv 0 \bmod p\AF$. The element $x\in \AF\subseteq \AK$ is integral over $\AFp$ and hence an element of $C$. Since $x \equiv 0 \bmod p\AF$, $\overline{x}\in C/pC \cong\widehat{C}/p\widehat{C}$ is contained in the kernel of the map $\widehat{C}/p\widehat{C} \rightarrow \mathbb{E}_K$. We prove that $x$ is not zero in $C/pC \cong \widehat{C}/p\widehat{C}$.

Since $x \equiv 0 \bmod p\AF$, $x = py$ for some $y$ in $\AF$. Then $y$ satisfies $f(py) = p^2 y^2 - \mu p y + p = 0$ in $\AF$, and since $p$ is not a zero divisor this simplifies to $p y^2 - \mu y + 1 = 0$. We then see $y\equiv \mu^{-1}\bmod p\AF$. In particular, $y\notin C$: if it were, the reduction modulo $p\AF$ would remain integral over $\mathbb{E}_F^+$, which is integrally closed in $\mathbb{E}_F$. It also follows that the class of $x$ is nonzero in $C/pC$: if it were zero, $x=pz$ for some $z$ in $C$ and then $py=pz \in \AK$ forcing $y=z$ as $p$ is not a zero divisor, but $y$ is not in $C$. Therefore, $x$ is a nonzero element in $C/pC \cong \widehat{C}/p\widehat{C}$.

\section*{Appendix 2: A result of Dylan Pentland}
This section is due to Dylan Pentland\footnote{\textit{Email address}: \url{dpentland@math.harvard.edu}}. I thank him for allowing me to include it in this paper.

Let $F=\mathbb{Q}_p$ and $K=\mathbb{Q}_p(p^{1/e})$ where $\gcd(e,p(p-1))=1$ and $e>1$. In this appendix we describe a slight refinement of the discussion after Theorem 1.3 in \cite{Berger2014} to show that a subring $R \subset \mathbb{A}_K$ which is Frobenius-stable and satisfies $R/p \simeq \mathbb{E}_K^+$ does not exist for such $K$.
We will first need to obtain an explicit description of the field of norms of $K$ and the corresponding ring $\mathbb{A}_K$. 

\begin{lemma}
Let $F=\mathbb{Q}_p$ and $K=\mathbb{Q}_p(p^{1/e})$ where $\gcd(e,p(p-1))=1$ and $e>1$. Write $F_\infty$ for $\bigcup_n \mathbb{Q}_p(\zeta_{p^n})$, and similarly write $K_\infty$ for $\bigcup_n \mathbb{Q}_p(\zeta_{p^n},p^{1/e})$. Then $[K_\infty:F_\infty]=e$ and the residue field of $K_\infty$ is $\mathbb{F}_p$.
\label{lem:deg_e}
\end{lemma}

\begin{proof}
Write $F_n=\mathbb{Q}_p(\zeta_{p^n})$ and $K_n=\mathbb{Q}_p(\zeta_{p^n},p^{1/e})$. Then $F_n$ has residue field $\mathbb{F}_p$, and $F_n/\mathbb{Q}_p$ is totally ramified of degree $p^{n-1}(p-1)$. As $\gcd(e,p(p-1))=1$, it follows $K_n/F_n$ is totally ramified of degree $e$ and thus the residue field of each $K_n$ is $\mathbb{F}_p$. The same then follows for $K_\infty$.

To see $[K_\infty:F_\infty]=e$, since $[K_n:F_n]=e$ it follows $[K_\infty:F_\infty] \le e$. To see it is equal to $e$, observe $K_\infty = F_\infty(p^{1/e})$, so we must show the minimal polynomial of $p^{1/e}$ over $F_\infty$ has degree $\ge e$. The minimal polynomial belongs to $F_n[X]$ for some $n\geq 1$. As $[K_n:F_n]=e$ it must have degree at least $e$. 
\end{proof}

Recall that $\AF= \left(\mathbb{Z}_p[[\mu]][1/\mu]\right)^\wedge_p$, where $\mu=[\epsilon]-1$.
\begin{lemma}\label{lem:field_norms_tame}
Let $K=\mathbb{Q}_p(p^{1/e})$ where $\gcd(e,p(p-1))=1$ and $e>1$. Then we have
\[\mathbb{A}_K = \left(\mathbb{Z}_p[[\mu_K]][1/\mu_K]\right)^{\wedge}_p\]
for some $\mu_K\in W(\widehat{K_\infty}^\flat)$ and $\mathbb{A}_K/p=\mathbb{F}_p((\pi_K))$ with $\mu_K^e = \mu$ and $\pi_K=\overline{\mu_K}$. Moreover, 
\[\varphi(\mu_K)^e=(1+\mu_K^e)^p-1.\]
\end{lemma}
\begin{proof}
Let $\mathbb{E}_{\mathbb{Q}_p}=\mathbb{F}_p((\varepsilon-1))$ and $\mathbb{E}_K$ be the fields of norms of $F_\infty/\mathbb{Q}_p$ and $K_\infty/K$ in the notation of Lemma \ref{lem:deg_e}. By the Fontaine--Wintenberger theorem (\cite[Th\'{e}or\`{e}me 3.1.6]{Fontaine1990}, \cite{Wintenberger1983}) and Lemma \ref{lem:deg_e}, the extension $\mathbb{E}_K/\mathbb{E}_{\mathbb{Q}_p}$ is finite separable, totally ramified, and of degree $e$ (note $K_\infty$ and $\mathbb{E}_K$ have the same residue field under the correspondence). We now identify the extension.

Let $\pi_K'\in \mathbb{E}_{K}^+$ be an arbitrary uniformizer. Then, there exists a unit $u\in (\mathbb{E}_K^+)^\times$ such that $(\pi_K')^e u= \epsilon-1$ in $\mathbb{E}_K^+$. Consider a $1$-variable polynomial $X^e-u$ in $\mathbb{E}_{K}^+[X]$. Recall that the residue field of $\mathbb{E}_K^+$ is $\mathbb{F}_p$. Since $u$ is a unit, its image $\overline{u}\in \mathbb{F}_p$ is also a unit. As $\gcd(e, p-1)=1$, $\overline{u}$ has an $e$th root in $\mathbb{F}_p$. Since $\gcd(e,p)=1$, we can apply Hensel's lemma to $X^e-u$ and obtain a root $w\in \mathbb{E}_K^+$. Then, $\pi_K:=  \pi_K'w\in \mathbb{E}_K^+$ is a uniformizer satisfying $\pi_K^e=\epsilon-1$. By the construction of $\mathbb{A}_K$ described after \cref{lem:main thm for unramified extension}, we see that $\mathbb{A}_K = \left(\mathbb{Z}_p[[\mu_K]][1/\mu_K]\right)^{\wedge}_p$ where $\mu_K\in W(\widehat{K_\infty}^\flat)$ satisfying $\mu_K^e = \mu$.

Finally, on $\mathbb{A}_{\mathbb{Q}_p}$ the Frobenius is $\varphi(\mu)=(1+\mu)^p-1$. Since $\mu=\mu_K^e$, its extension to $\mathbb{A}_K$ satisfies 
\[\varphi(\mu_K)^e=\varphi(\mu_K^e)=\varphi(\mu)=(1+\mu)^p-1=(1+\mu_K^e)^p-1.\] 
\end{proof}

We may now deduce the desired claim from this description of $\mathbb{A}_K$.

\begin{proposition}\label{prop:Deylan Pentland}
Let $K=\mathbb{Q}_p(p^{1/e})$ where $\gcd(e,p(p-1))=1$ and $e>1$. Then there is no subring $R$ of $\mathbb{A}_K$ which is stable under the Frobenius map and with the property that 
\[R/p \simeq \mathbb{E}_K^+ \subset \mathbb{E}_K\] 
via the natural map $R/p \rightarrow \mathbb{A}_K/p=\mathbb{E}_K$.
\end{proposition}

\begin{proof}
Assume for the sake of contradiction that such a subring $R$ exists. We assume $R/p \simeq \mathbb{E}_K^+ \subset \mathbb{E}_K$, so it must be the case that $R \cap p \mathbb{A}_K = p R$ (in particular the subring is $p$-saturated in $\mathbb{A}_K$). Indeed, by assumption the natural map $R/p \to \mathbb{A}_K/p$ is injective, which means if $x\in R$ and $x\in p \mathbb{A}_K$ then $x\in p R$. Thus $R \cap p \mathbb{A}_K \subset pR$. The reverse inclusion is clear, so $R \cap p \mathbb{A}_K = p R$. In what follows, we use the notation of Lemma \ref{lem:field_norms_tame} to identify $\mathbb{E}_K$ and $\mathbb{A}_K$.

Since $\pi_K$ is in $\mathbb{E}_K^+$, there is a lift $x=\mu_K + pf \in R$ where $f\in \mathbb{A}_K$. Since we assumed $R/p \cong \mathbb{E}_K^+ \subseteq \mathbb{E}_K \subseteq \widehat{K_\infty}^\flat$, we have $p W(\widehat{K_\infty}^\flat) \cap R= p R$ by a similar argument as above. As $\varphi\colon W(\widehat{K_\infty}^\flat) \rightarrow W(\widehat{K_\infty}^\flat)$ is a Frobenius lift and we assumed that $R$ is stable under the Frobenius map, it follows that $\varphi(r)-r^p\in p W(\widehat{K_\infty}^\flat) \cap R=p R$ for all $r\in R$.
Since $R \subseteq \mathbb{A}_K$ is $p$-torsion free, $\delta(r) = (\varphi(r)-r^p)/p$ is in $R$ for all $r\in R$, i.e., we get $\delta(-) = (\varphi(-)-(-)^p)/p$ as a $\delta$-ring structure on the subring $R$. 
It must be the case that $\delta(x)$ reduces modulo $p\mathbb{A}_K$ to an element of $\mathbb{E}_K^+=\mathbb{F}_p[[\pi_K]]$, as $\delta(x)\in R$.

By the construction of $\delta(\mu_K)$ we have $\varphi(\mu_K)=\mu_K^p + p \delta(\mu_K)$, and thus expanding we obtain 
\[\varphi(\mu_K)^e \equiv \mu_K^{pe} + ep \delta(\mu_K) \mu_K^{p(e-1)} \pmod{p^2\mathbb{A}_K}.\] 
We also have $\varphi(\mu_K)^e = (1+\mu_K^e)^p-1$ by Lemma \ref{lem:field_norms_tame}, so expanding this as well we arrive at
\begin{align*}
    \delta(\mu_K) \equiv e^{-1} \sum_{i=1}^{p-1} \frac{1}{p} \binom{p}{i} \mu_K^{ei-p(e-1)} \pmod{p\mathbb{A}_K}.
\end{align*}  
The $i=1$ term in the sum contributes $e^{-1} \mu_K^{e-p(e-1)}$, which has a $\mu_K$-exponent that is $< 0$ when $e>1$ and $\gcd(e,p)=1$ (as we assume). We obtain $x^p \equiv \mu_K^p \bmod p^2\mathbb{A}_K$ by expanding $x$ as $x=\mu_K + pf$, so that 
\begin{align*}
    \delta(x) = \frac{\varphi(x)-x^p}{p} \equiv \frac{\varphi(\mu_K)+p \varphi(f) - \mu_K^p}{p} = \delta(\mu_K) + \varphi(f) \pmod{p\mathbb{A}_K}.
\end{align*} 

Thus modulo $p\mathbb{A}_K$, the $i=1$ term in the identity $\delta(\mu_K) \equiv e^{-1} \sum_{i=1}^{p-1} p^{-1} \binom{p}{i} \mu_K^{ei-p(e-1)} \pmod{p\mathbb{A}_K}$ still contributes a negative exponent not divisible by $p$. We also note that the coefficient of the $i=1$ term is $e^{-1}$, and in particular nonzero. 

This term cannot be canceled by $\varphi(f)$ modulo $p\mathbb{A}_K$ as this can only have exponents divisible by $p$ (by virtue of living in the image of $\varphi: \mathbb{F}_p((\pi_K)) \rightarrow \mathbb{F}_p((\pi_K))$ sending $\pi_K \rightarrow \pi_K^{p}$). It follows that $\delta(x)$ modulo $p \mathbb{A}_K$ is not an element of $\mathbb{E}_K^+=\mathbb{F}_p[[\pi_K]]$, which is a contradiction.
\end{proof}

\bibliographystyle{halpha}
\bibliography{A_Kplus}
\end{document}